\documentclass[12pt, reqno, twoside, letterpaper]{amsart}

\usepackage{paperstyle}

\def\tsL{\widetilde \sL}

\title[Quadratic Fields Generated by the Shanks Sequence]
{On the Number of Distinct  Quadratic Fields\\ Generated by the Shanks Sequence}

\author[W.\ D.\ Banks]{William D.\ Banks}

\address{Department of Mathematics, 
         University of Missouri, 
         Columbia, MO 65211, USA.}

\email{bankswd@missouri.edu}
      
\author[I.\ E.\ Shparlinski]{Igor E.\ Shparlinski}

\address{Department of Pure Mathematics,
		 University of New South Wales,
		 Sydney, NSW 2052, Australia.}
		 
\email{igor.shparlinski@unsw.edu.au}
  
\date{\today}

\begin{document}

\begin{abstract}
Let $g>1$ be an integer and $f(X)\in{\mathbb Z}[X]$ a polynomial 
of positive degree with no multiple roots, and put $u(n)=f(g^n)$.
In this note, we study the sequence of quadratic fields
${\mathbb Q}\(\sqrt{u(n)}\,\)$
as $n$ varies over the consecutive integers
$M+1,\ldots,M+N$.  Fields of this
type include Shanks fields and their generalizations.
Using the square sieve together with new bounds on character
sums, we improve an upper bound of Luca and Shparlinski (2009)
on the number of $n \in \{M+1,\ldots,M+N\}$ with ${\mathbb Q}\(\sqrt{u(n)}\,\) = {\mathbb Q}\(\sqrt{s}\,\)$ for a given squarefree integer $s$. 
\end{abstract}


\maketitle

\section{Introduction}

Let $f(X)\in\ZZ[X]$ be a polynomial of degree $d\ge 1$ that has
a positive leading coefficient and no multiple roots
in its splitting field.  Let $g>1$ be a fixed integer, and put
\begin{equation}
\label{eq:undefn}
u(n)=f(g^n)\qquad(n\ge 1).
\end{equation}
Note that $u(n)>0$ for all but finitely many $n$.
In this paper, we study those quadratic fields
$\QQ\(\sqrt{u(n)}\,\)$ that arise 
as $n$ varies over a sequence of consecutive integers.

Motivated by work of Shanks~\cite{Shan}, which
corresponds to the particular case in which $u(n)=(2^n+3)^2-8$
and certain generalizations (see~\cite{PatvdPWill,vdPWill,Will}), 
Luca and Shparlinski~\cite{LuSh2} have studied the distribution 
of the quadratic fields $\QQ\(\sqrt{u(n)}\,\)$ in the
general setting of~\eqref{eq:undefn}.  To describe the results,
let $Q_u(M,N;s)$ be the number of integers $n\in[M+1,M+N]$
for which $\QQ\(\sqrt{u(n)}\,\)=\QQ(\sqrt{s}\,)$,
where $M,N,s\in\ZZ$ with $N\ge 1$ and $s$ squarefree.
Using the square sieve of Heath-Brown~\cite{HB1}
along with some knowledge about prime divisors of shifted primes,
in~\cite[Theorem~1.3]{LuSh2} a nontrivial upper bound on 
$Q_u(M,N;s)$ is given.

More precisely, let  $\tau_\ell$ denote
the multiplicative order of $g$ modulo a prime $\ell$, that is,
the smallest positive integer $\tau$ for which $g^\tau\equiv 1\bmod \ell$.
Let $\alpha_0$ be a fixed real number for which
one has a lower bound of the form 
\begin{equation}
\label{eq:good alpha0}
\left|\left\{\ell \le z:\ell~\text{is prime and~}
\tau_\ell \ge \ell^{\alpha_0}\right\}\right| \gg \frac{z}{\log z}
\end{equation}
for all sufficiently large $z$, where the implied
constant depends only on $\alpha_0$ (see \S\ref{sec:gen not}
for the definitions of $\gg$, $\ll$ and other related symbols).
We also let $\alpha$ a fixed real number for which
one has a lower bound of the form 
\begin{equation}
\label{eq:good alpha}
\left|\left\{\ell \le z:\ell~\text{is prime and~}
P^+(\ell-1)\ge \ell^\alpha\right\}\right| \gg \frac{z}{\log z}
\end{equation}
for all sufficiently large $z$, where $P^+(k)$ denotes the largest
prime divisor of an integer $k\ge 2$, and the implied constant
depends only on $\alpha$ 
Using a result of Baker and Harman~\cite{BaHa} one can take
\begin{equation}
\label{eq:BH bound} 
\alpha_0 \ge \alpha =0.677,
\end{equation}
and under the Extended Riemann Hypothesis (ERH) one can take any
real number $\alpha_0<1$; see~\cite{ErdMur,Papp}.  It is straightforward 
to show that if $\alpha>1/2$ is admissible for~\eqref{eq:good alpha} then
$\alpha_0=\alpha$ is also admissible for~\eqref{eq:good alpha0}; in fact,
this is the only known approach for getting large values of $\alpha_0$
unconditionally. However, under the ERH 
every value  $\alpha_0<1$ is admissible for~\eqref{eq:good alpha0}
in a very strong sense (see Lemma~\ref{lem:Mult Ord ERH}
below), but the values of $\alpha$ are not improved under the ERH.

In the above notation, under natural conditions as in our 
Theorem~\ref{thm:QsMN} below, in~\cite[Theorem~1.3]{LuSh2}
it is shown that the bound
\begin{equation}
\label{eq:QuMNsest}
Q_u(M,N;s) \ll N^{\beta_0}(\log N)^{\gamma_0}
\end{equation}
holds uniformly for all choices of $M,N,s$ as above, where
$$
\beta_0  = \frac{3}{2(1+\alpha_0)}\mand
\gamma_0=\frac{4+\alpha_0}{1+\alpha_0}.
$$
(in~\cite{LuSh2} the results are formulated in terms 
of $\alpha$, however the argument only depends on the parameter $\alpha_0$).
Thus, using~\eqref{eq:BH bound} we see that~\cite[Theorem~1.3]{LuSh2} yields~\eqref{eq:QuMNsest}
unconditionally with
\begin{equation}
\label{eq:BH} 
\beta_0  = 0.89445 \cdots
\end{equation}
and conditionally (under ERH) with
\begin{equation}
\label{eq:ERH} 
\beta_0= 0.75+\eps
\end{equation}
for any fixed $\eps>0$. 

\begin{theorem}
\label{thm:QsMN}
Let $g>1$ be an integer, and let
$f(X)\in\ZZ[X]$ be a polynomial of degree $\deg f\ge 3$
with a positive leading coefficient and no multiple roots
in its splitting field.  Put $u(n)=f(g^n)$.
If $\alpha$ satisfies \eqref{eq:good alpha} then 
uniformly for $M,N,s\in\ZZ$ with $N\ge 1$ and $s$ squarefree,
the upper bound
$$
Q_u(M,N;s) \ll  N^\beta(\log N)^\gamma
$$
holds with
$$
\beta=(2\alpha)^{-1}\mand\gamma=2-\alpha^{-1},
$$
where the implied constant depends only on $f$, $g$ and $\alpha$. 
\end{theorem}

Consequently, in place of~\eqref{eq:BH} we get the value
$$
\beta = 0.73855\cdots,
$$
which is unconditionally sharper than \eqref{eq:ERH}. 
Note that Theorem~\ref{thm:QsMN} immediately implies the
lower bound $c N^{1-\beta}(\log N)^{-\gamma}$ (with a 
constant $c > 0$ that depends on $\alpha$)
for the number of \emph{distinct} quadratic fields
$\QQ\(\sqrt{u(n)}\,\)$ that arise as $n$ varies over
the interval $[M+1,M+N]$.

Our next result improves the bound of Theorem~\ref{thm:QsMN}
on average over $s$.
Let 
\begin{equation}
\label{eq:sQuMNSdefn}
\sQ_u(M,N;S)=\sum_{\substack{s\le S\\s\text{~sqfree}}} 
Q_u(M,N;s).
\end{equation}
The uniformity with respect to $s$ in the bound of Theorem~\ref{thm:QsMN} implies that 
$$
\sQ_u(M,N;S)\ll SN^{1/(2\alpha)+o(1)}\qquad(N\to\infty).
$$
This can be strengthened as follows. 

\begin{theorem}
\label{thm:QsMN Av1} 
In the notation of Theorem~\ref{thm:QsMN}
and of~\eqref{eq:sQuMNSdefn}, we have
$$
\sQ_u(M,N;S)\le\(SN^{1/(1+\alpha)}
+S^{1-1/(4\alpha)}N^{1/(2\alpha)}\)N^{o(1)}
$$
as $N\to\infty$,
where the function of $N$ implied by $o(1)$ depends only on $f$, $g$ and $\alpha$.
\end{theorem}

Furthermore, using a slightly different approach we show that 

\begin{theorem}
\label{thm:QsMN Av2}
In the notation of Theorem~\ref{thm:QsMN}
and of~\eqref{eq:sQuMNSdefn}, we have
\begin{equation}
\label{eq:twoterms}
\sQ_u(M,N;S)\le \(S^{3/(3+\alpha)}N^{(3-\alpha)/(3+\alpha)}
+S^{1/2}N^{(3-\alpha)/(1+3\alpha)}\)N^{o(1)}
\end{equation}
as $N\to\infty$,
where the function of $N$ implied by $o(1)$ depends only on $f$, $g$ and $\alpha$.
\end{theorem}

One verifies that the second term in the bound of Theorem~\ref{thm:QsMN Av1}
dominates the first one for small values of $S$ and the switching point  at $S =  N^{2(1-\alpha)/(1+\alpha)}$
is above the value of $S =  N^{2(1-\alpha)/(1+3\alpha)}$, 
where the bound of  Theorem~\ref{thm:QsMN Av2} become stronger.  Thus, 
straightforward calculations show that
Theorems~\ref{thm:QsMN Av1} and~\ref{thm:QsMN Av2}
can be combined into the following statement:

\begin{cor}
\label{cor:QsMN Av-Comb}
In the notation of Theorem~\ref{thm:QsMN}
and of~\eqref{eq:sQuMNSdefn}, we have
\begin{align*}
\sQ_u(M&,N;S)\\
&\le \begin{cases}
S^{1-1/(4\alpha)}N^{1/(2\alpha)+o(1)}
&\quad\hbox{if $S\le N^{2(1-\alpha)/(1+3\alpha)}$},\\
S^{1/2}N^{(3-\alpha)/(1+3\alpha)+o(1)}
&\quad\text{if $N^{2(1-\alpha)/(1+3\alpha)}< S\le N^{4(1-\alpha)/(1+3\alpha)}$},\\
S^{3/(3+\alpha)}N^{(3-\alpha)/(3+\alpha)+o(1)}
&\quad\text{if $ N^{4(1-\alpha)/(1+3\alpha)} < S\le N^{2\alpha/3}$},
\end{cases}
\end{align*}
as $N\to\infty$,
where the functions of $N$ implied by $o(1)$ depend only on $f$, $g$ and $\alpha$.
\end{cor}

We note that the end point in the third  range of Corollary~\ref{cor:QsMN Av-Comb}
is unnecessary; it is given to indicate the largest value of $S$ for which we
have an improvement over the trivial bound $\sQ_u(M,N;S) \le N$. 

The proofs of  Theorems~\ref{thm:QsMN Av1} and~\ref{thm:QsMN Av2}
involve sums with Jacobi symbols over integers $s \in [1,S]$. We use a result 
of Heath-Brown~\cite[Corollary~3]{HB2} to estimate such sums ``on-average,''
but on several occasions we need ``individual'' estimates, and in those
instances we use only the trivial bound $S$ to estimate the sums.  This does
not involve any substantial sacrifice, however, since in the most interesting ranges
these sums are shorter than the range covered by the 
Polya-Vinogradov and Burgess  bounds (see~\cite[Theorems~12.5 and 12.6]{IwKow}).
On the other hand, under the ERH, not only can we use any $\alpha<1$ 
in~\eqref{eq:good alpha}, but we can also exploit a square root cancellation in
character sums (see Lemma~\ref{lem:ERH} below),
which leads to a much better estimate. 

\begin{theorem}
\label{thm:QsMN Av ERH}
Under the ERH, in the notation of Theorem~\ref{thm:QsMN}
and of~\eqref{eq:sQuMNSdefn}, we have
$$
\sQ_u(M,N;S)\le S ^{1/2} N^{1/(1+\alpha)+o(1)}
$$
as $N\to\infty$,
where the function of $N$ implied by $o(1)$ depends only on $f$, $g$ and $\alpha$.
\end{theorem}

Examining the proof of Theorem~\ref{thm:QsMN Av ERH} we see that it also yields the 
bound
$$
Q_u(M,N;s) \ll  N^{1/(1+\alpha)+o(1)},
$$
which with $\alpha$ as in~\eqref{eq:BH bound} becomes 
$$
Q_u(M,N;s) \le N^{1000/1677+o(1)};
$$
this improves~\eqref{eq:ERH}  and the unconditional bound of Theorem~\ref{thm:QsMN}
(note that $1000/1677 = 0.596302 \cdots $).

We remark that Cutter, Granville and Tucker~\cite[Theorems~1A and~1B]{CGT}
have obtained an asymptotic
formula for the number of distinct fields of the form $\QQ\(\sqrt{f(n)}\,\)$ with
$n=1,\ldots,N$, where $f(X)\in\ZZ[X]$ is a given polynomial
of degree at most two. For  polynomials $f(X)$ of degree three
or more, a conditional asymptotic formula based on the $ABC$-conjecture
is given in~\cite[Theorem~1C]{CGT}; see also~\cite{LuSh1}.

\section{Preliminaries}

\subsection{General notation} 
\label{sec:gen not}
For an odd integer $m$, we use $(\tfrac km)$ to
denote the Jacobi symbol of $k$ modulo~$m$. 

We also use $\varphi(k)$ to denote the Euler function
of an integer $k \ge 1$. 

As usual, we write $\e(t)= \exp(2 \pi i t)$ for all $t\in\RR$.

For a fixed nonzero integer $\lambda$ and any integer $m$
that is coprime to $\lambda$, we use $\tau_m(\lambda)$ to denote
the multiplicative order of $\lambda$ modulo $m$, that is, 
the smallest positive integer $\tau$ for which
$g^\tau\equiv 1\bmod p$.

Throughout the paper, we use the symbols $O$, $o$, $\ll$, $\gg$
and $\asymp$ along with their standard meanings; any constants or
functions implied by these symbols may depend on the fixed
polynomial $f(X)\in\ZZ[X]$ or the parameter $\alpha$ but are
independent of other variables except where indicated.

\subsection{Auxiliary results}

In \S\ref{sec:proofs} we use the following technical lemma; for a proof, see
Graham and Kolesnik~\cite[Lemma~2.4]{GraKol}.

\begin{lemma}
\label{lem:grakol}
Let
$$
\cB(z)=\sum_{j=1}^J A_j z^{B_j}+\sum_{k=1}^K C_k z^{-D_k},
$$
where $A_j,B_j,C_k,D_k>0$.  For any $z_2\ge z_1>0$ there exists $z\in[z_1,z_2]$ such that
$$
\cB(z)\ll \sum_{j=1}^J \sum_{k=1}^K T_{jk} 
+\sum_{j=1}^J A_j z_1^{B_j}+\sum_{k=1}^K C_k z_2^{-D_k},
$$
where
$$
T_{jk}  = \(A_j^{D_k} C_k^{B_j}\)^{1/(B_j+D_k)}
\qquad (1\le j\le J,~1\le k\le K).
$$
\end{lemma}

As usual, we use $\pi(t;m,a)$ to denote the number of primes
$p\le t$ for which $p \equiv a \bmod m$.
We apply the Brun-Titchmarsh theorem in the following relaxed form
(see~\cite[Theorem~6.6]{IwKow} for 
 a much more precise statement). 
 
 \begin{lemma}
\label{lem:B-T}
Fix $\eta > 0$. For any real number $t\ge 2$ and integer
$m\le t^{1-\eta}$, we have 
$$
\pi(t; m, a) \ll \frac{t}{\varphi(m)\log t},
$$
where the implied constant depends only on $\eta$. 
\end{lemma}

We also need the following estimate. 

\begin{lemma}
\label{lem:Euler recipr}
For any real number $t\ge 2$, we have 
$$
\sum_{n\le t}\frac{n}{\varphi(n)^2}\ll\log t.
$$
\end{lemma}

\begin{proof}
If $f$ is any multiplicative function
satisfying 
\begin{equation}
\label{eq:wirscond}
0\le f(p^k)=1+\frac{c}{p}+O(p^{-2})\qquad(p~\text{prime},~k\ge 1),
\end{equation}
where the constant $c$ and that implied by the $O$-symbol
depend only on $f$, then as a special case of the well known
theorem of Wirsing~\cite{Wirs} one sees that
$\sum_{n\le t}f(n)\ll t$. Applying this result with the function
$f(n)=n^2/\varphi(n)^2$ (which verifies~\eqref{eq:wirscond}
with $c=2$), we deduce the bound
$$
\sum_{n\le t}\frac{n^2}{\varphi(n)^2}\ll t.
$$
The stated result follows from this by partial summation.
\end{proof}

One can easily obtain an asymptotic formula for the sum in 
Lemma~\ref{lem:Euler recipr}, but the upper bound is
quite sufficient for our purposes here. 

\subsection{Multiplicative orders}

We recall the following result of  Erd\H os and  Murty \cite[Theorem~4]{ErdMur} 
in a slightly weakened form. 

\begin{lemma}
\label{lem:Mult Ord ERH} Under the ERH, for any fixed integer $g>1$ the
inequality $\tau_\ell(g)>\ell/\log\ell$ holds for all primes $\ell\le z$
with at most $o(z/\log z)$ exceptions as $z\to\infty$.
\end{lemma}

From this we immediately derive the next statement.

\begin{cor}
\label{cor:Mult Ord ERH} Under the ERH, for any fixed integer $g >1$  and any  fixed $\alpha>1/2$ which is 
admissible for~\eqref{eq:good alpha} we have  
$$
\left|\left\{\ell \le z:\ell~\text{is prime~} \tau_\ell(g) > \ell/\log \ell~\text{and~}
P^+(\ell-1)\ge \ell^\alpha\right\}\right| \gg \frac{z}{\log z} \qquad (z \to \infty).
$$ 
\end{cor}

\section{Character sums}

\subsection{Bounds on character sums with exponential functions}
\label{sec:char sum}

In this section only, we write $\tau_m$ for $\tau_m(\lambda)$ to simplify
the notation.

For the proof of Theorem~\ref{thm:QsMN}, we need some bounds
for character sums.

We use the following variant of the result of Korobov~\cite[Theorem~3]{Kor}.
We present it here in a simplified form which is suited to our applications and 
can be extended in several directions. 

\begin{lemma}
\label{lem:Prod Form} Let $f(X) \in \ZZ[X]$, and let $\lambda\in\ZZ$,
$\lambda\ne 0$. Let   $\ell,p$  be distinct primes with 
$$
\gcd(\ell p, \lambda) = \gcd(\tau_\ell, \tau_p) = 1.
$$ 
For any integer $a$ we define integers $a_\ell$ and $a_p$ by the conditions
$$
a_\ell \tau_p + a_p \tau_\ell \equiv a \bmod {\tau_{\ell p}},
\qquad 0 \le a_\ell <  \tau_\ell,\qquad \ 0 \le a_p < \tau_p.
$$
Then,
\begin{align*}
\sum_{n=1}^{\tau_{\ell  p}}\(\frac{f(\lambda^n)}{\ell   p}\) &\e(an/\tau_{\ell  p})
= \sum_{x=1}^{\tau_\ell}\(\frac{f(\lambda^x)}{\ell}\) \e(a_\ell x/\tau_\ell)  
\sum_{y=1}^{\tau_p} \(\frac{f(\lambda^y)}{p}\) \e(a_p y/\tau_p).
\end{align*}
\end{lemma}

\begin{proof} We follow closely the proof of~\cite[Theorem~3]{Kor}. 
Using the coprimality condition   $\gcd(\tau_\ell, \tau_p)=1$ we see that  
the integers 
$$
x \tau_p + y \tau_\ell,\qquad  0 \le x < \tau_\ell,\qquad \ 0 \le y < \tau_p, 
$$
run through the complete residue system modulo $\tau_{\ell p} = \tau_\ell \tau_p$.
Moreover,
$$
\lambda^{x \tau_p + y \tau_\ell} \equiv \lambda^{x \tau_p} \bmod \ell,\qquad  
\lambda^{x \tau_p + y \tau_\ell} \equiv \lambda^{y \tau_\ell} \bmod p,
$$
and
$$
 \e(a (x \tau_p + y \tau_\ell)/\tau_{\ell  p})= 
  \e (ax/\tau_\ell)\,\e(a y/\tau_p), 
$$
Hence, using the multiplicativity of the Jacobi symbol, we have
\begin{align*}
&\sum_{n=1}^{\tau_{\ell  p}}\(\frac{f(\lambda^n)}{\ell   p}\) \e(an/\tau_{\ell  p})  \\
&  \qquad\qquad = \sum_{x=1}^{\tau_\ell}
 \sum_{y=1}^{\tau_p} \(\frac{f(\lambda^{x \tau_p + y \tau_\ell})}{\ell}\)    
  \(\frac{f(\lambda^{x \tau_p + y \tau_\ell} )}{p}\)  \e(a (x \tau_p + y \tau_\ell)/\tau_{\ell  p})\\
  &  \qquad\qquad = \sum_{x=1}^{\tau_\ell}
\(\frac{f(\lambda^{x \tau_p})}{\ell}\)     \e (ax/\tau_\ell )
 \sum_{y=1}^{\tau_p}  \(\frac{f(\lambda^{y \tau_\ell} )}{p}\)   \e (a y/\tau_p).
\end{align*}
Replacing $x$ with $x \tau_p^{-1} \bmod {\tau_\ell}$
and $y$ with $y \tau_\ell^{-1} \bmod {\tau_p}$,
and taking into account that
$a \tau_p^{-1} \equiv a_\ell \bmod {\tau_\ell}$
and
$a \tau_\ell^{-1} \equiv a_p \bmod {\tau_p}$,
the result follows. 
\end{proof}

The next statement follows immediately from the Weil bound 
on sums with multiplicative characters; see, for example,~\cite[Theorem~11.23]{IwKow}.

\begin{lemma}
\label{lem:Char fgu p}
Let $f(X) \in \ZZ[X]$ be monic of degree $d\ge 1$ with no multiple roots
in its splitting field, and let $\lambda\in\ZZ$, $\lambda\ne 0$.
For any prime $p$ coprime to $\lambda f(0)$ and any integer~$a$,
we have
$$
\sum_{x=1}^{\tau_p}
\(\frac{f(\lambda^{x})}{p}\)\e(ax/\tau_p)\ll p^{1/2}.
$$
\end{lemma}

\begin{proof} Denoting $s = (p-1)/\tau_p$, 
we can write $\lambda= \vartheta^s$ with \emph{some} primitive root $\vartheta$ modulo $p$. 
Then,
\begin{equation}
\label{eq:Expand}
\begin{split}
\sum_{x=1}^{\tau_p} \(\frac{f(\lambda^{x})}{p}\)\e(ax/\tau_p) 
&=\frac{1}{\tau_p}\sum_{x=1}^{p-1} \(\frac{f(\vartheta^{sx})}{p}\)\e(asx/(p-1))\\
&=\frac{1}{s}\sum_{w=1}^{p-1}\(\frac{f(w^s)}{p}\)\chi(w),
\end{split}
\end{equation}
where $\chi$ is the multiplicative character modulo $p$ defined by
$$
\chi(w)=\e(asx/(p-1))\qquad(w\in\ZZ,~p\nmid w),
$$
where $x$ is any integer for which $w\equiv\vartheta^x\bmod p$.

Let $g(X) = f'(X)$ be the derivative of $f$. 
Since $f(X)$ has no multiple roots, $f(0) \not \equiv  0 \bmod p$, and
$$
 \frac{d}{dX}\,f(X^s)=sX^{s-1}g(X^s),
$$
we see that $f(X^s)$ has no multiple roots (and zero is not a root of $f(X^s)$). 
Thus, the Weil  bound in the form given by~\cite[Theorem~11.23]{IwKow} shows that
$$
\sum_{w=1}^{p-1}\(\frac{f(w^s)}{p}\)\chi(w)\ll p^{1/2}.
$$
Using~\eqref{eq:Expand} the result now follows. 
\end{proof}

Combining Lemmas~\ref{lem:Prod Form} and~\ref{lem:Char fgu p}, 
we obtain the following statement.

\begin{lemma}
\label{lem:Char fgu pl} 
Let $f(X) \in \ZZ[X]$ be monic of degree $d\ge 1$ with no multiple roots
in its splitting field, and let $\lambda\in\ZZ$, $\lambda\ne 0$.
Let   $\ell,p$  be distinct primes with 
$$
\gcd(\ell p, \lambda f(0)) = \gcd(\tau_\ell, \tau_p) = 1.
$$ 
Then,
$$
\sum_{n=1}^{\tau_{\ell p} }\(\frac{f(\lambda^n)}{\ell   p}\)
\e(an/\tau_{\ell  p})\ll(\ell p)^{1/2}.
$$
\end{lemma}

Using Lemma~\ref{lem:Char fgu pl} we derive the following statement,
which is our principal technical tool; this result
improves upon~\cite[Lemma~4.1]{LuSh2} but requires that the additional
coprimality condition $\gcd(\tau_\ell,\tau_p)=1$ is met; see also~\cite{DobWil,Yu}
for some similar bounds with linear polynomials.

\begin{lemma}
\label{lem:Char fgu pl N}
Let $f(X) \in \ZZ[X]$ be monic of degree $d\ge 1$ with no multiple roots
in its splitting field, and let $\lambda\in\ZZ$, $\lambda\ne 0$.
Let   $\ell,p$  be distinct primes with 
$$
\gcd(\ell p, \lambda f(0)) = \gcd(\tau_\ell, \tau_p) = 1.
$$ 
Then, for any integer  $A$ with $\gcd(A, \ell p) =1$ and $K \ge 1$,  we have
$$
\sum_{n=1}^{K}\(\frac{f(A \lambda^k)}{\ell   p}\) 
 \ll \frac{K(\ell   p)^{1/2}}{\tau_{\ell p}}+ (\ell   p)^{1/2} \log
(\ell   p).
$$
\end{lemma}

The proof of Lemma~\ref{lem:Char fgu pl N}
(which we omit) uses Lemma~\ref{lem:Char fgu pl} in conjunction
with the standard technique of deriving bounds on incomplete sums from
bounds on complete sums; see,  for example,~\cite[\S12.2]{IwKow}. 

\subsection{Sums with real characters}

\label{sec:RealChar}

We need the following  bound for character sums ``on
average'' over squarefree moduli, which is due to
Heath-Brown; see~\cite[Corollary~3]{HB2}.

\begin{lemma}
\label{lem:HB} For positive integers $R,S$ and 
a function $\psi:\RR_{>0}\to\CC$ we have
$$
\sum_{\substack{m\le R\\m\text{~odd sqfree}}}\left|
\sum_{s\le S}\psi(s)\(\frac{s}{m}\)\right|^2\le
S(R+S)(RS)^{o(1)}\max_{1\le s\le S}~\left|\psi(s)\right|^2
$$
as $\max\{R,S\}\to\infty$, where the function of $R,S$ implied
by $o(1)$ depends only on $\psi$.
\end{lemma}

\subsection{Character sums under the ERH} 

Under the ERH we have  
the following well-known estimate (see~\cite[\S1]{MoVau}; it
can also be derived from~\cite[Theorem~2]{GrSo}).

\begin{lemma}
\label{lem:ERH} For integers $q>k\ge 1$ and a primitive 
character $\chi$ modulo $q$, the bound
$$
\left|\sum_{n\le k}\chi(n)\right|\le k^{1/2} q^{o(1)}
$$
holds, where the function implied by $o(1)$ depends only on $k$.
\end{lemma}

\section{The proofs}
\label{sec:proofs}

\subsection{Proof of Theorem~\ref{thm:QsMN}}

Initially, we follow the proof of~\cite[Theorem~1.3]{LuSh2}.  

For every $\alpha$ satisfying~\eqref{eq:good alpha}, there are 
constants $c>0$ and $C>1$ depending only on $\alpha$ with the following property.
For every sufficiently large real number $z>1$, there is a set $\sL_z$
containing at least $cz/\log z$ primes $\ell\in[z,Cz]$ for which
$$
\tau_\ell(g) \ge P^+(\ell-1) \ge z^\alpha\qquad(\ell\in \sL_z);
$$
see~\cite[Lemma~5.1]{LuSh2}.  Let $\alpha,c,C$ be fixed in what follows;
we can assume that $\alpha>\tfrac12$, as even the 
value~\eqref{eq:BH bound} is admissible.  We also assume that $z$ is large
enough so that the aforementioned property holds.

Let $\omega_z(k)$ be the number of distinct prime factors $\ell$ of
$k$ that lie in $\sL_z$.  Note that if $k\ge 1$ is a perfect square, then we always have
$$
\sum_{\ell \in \sL_z}\(\frac{k}{\ell}\)=|\sL_z|-\omega_z(k).
$$
Let $\sN_z$ be the set of integers $n \in [M+1, M+N]$ such that
$\omega_z(u(n))\le \tfrac12|\sL_z|$, and let $\sE_z$ be the set of
remaining integers $n \in [M+1, M+N]$. 
The following bound is~\cite[Equation~(6.1)]{LuSh2}:
\begin{equation}
\label{eq:Ez} |\sE_z| \ll N z^{-\alpha} +  \log z .
\end{equation}
For a fixed squarefree $s\ge 1$, let $\sN_{s,z}$ denote the set of
$n\in\sN_z$ for which $\QQ\(\sqrt{u(n)}\,\)=\QQ(\sqrt{s}\,)$.
For such $n$, it is clear that  
$su(n)$ is a perfect square, hence $s\mid u(n)$,
and $\omega_z(su(n)) = \omega_z(u(n))\le\tfrac12|\sL_z|$. 
Thus, for every $n\in\sN_{s,z}$ we have
$$
\sum_{\ell \in \sL_z} \(\frac{su(n)}{\ell}\) =|\sL_z| -
\omega_z(su(n)) \ge\tfrac12|\sL_z|.
$$
This implies that
\begin{equation}
\label{eq:nNsz} 
|\sN_{s,z}| \le \frac{2}{|\sL_z|}
\sum_{n \in \sN_{s,z}}
\left|   \sum_{\ell \in \sL_z} \(\frac{su(n)}{\ell}\)\right|^2.
\end{equation}
Using~\eqref{eq:Ez} and extending the summation in~\eqref{eq:nNsz} 
to all  $n \in [M+1, M+N]$ we deduce that
\begin{equation}
\begin{split}
\label{eq:basicQ} 
Q_u(M,N;s)  &  \ll   Nz^{-\alpha}+\log z+\frac{(\log z)^2}{z^2} \sum_{n = M+1}^{M+N} 
\left|\sum_{\ell \in \sL_z}\(\frac{su(n)}{\ell}\)\right|^2 . 
\end{split}
\end{equation}
Squaring out the right side and estimating the contribution from diagonal
terms $\ell=p$ trivially as $N$, we have 
\begin{equation*}
\begin{split}
 \sum_{n = M+1}^{M+N} 
\left|   \sum_{\ell \in \sL_z} \(\frac{su(n)}{\ell}\)\right|^2
=  \sum_{\ell, p\in \sL_z} \sum_{n =
M+1}^{M+N} \(\frac{su(n)}{\ell p}\) \ll N |\sL_z| + W,
\end{split}
\end{equation*}
where 
$$
W = \sum_{\substack{\ell, p\in \sL_z\\\ell \ne p} } \sum_{n =
M+1}^{M+N} \(\frac{su(n)}{\ell p}\).
$$
Since $\alpha<1$, the contribution to~\eqref{eq:basicQ} coming
from diagonal terms, namely, 
$$ 
N |\sL_z|\,\frac{(\log z)^2}{z^2} \ll N z^{-1} \log z,
$$
is dominated by the term $Nz^{-\alpha}$ and so can be dropped. This yields 
the bound
\begin{equation}
\label{eq:newQ} 
Q_u(M,N;s)    \ll   Nz^{-\alpha}+\log z  +
\frac{(\log z)^2}{z^2}W,
\end{equation}
which is a slightly simplified version of ~\cite[Equation~(6.4)]{LuSh2}.

Turning to the estimation of $W$, we now write $W=U+V$, where
\begin{equation*}
\begin{split}
U &= \sum_{\substack{\ell, p\in \sL_z,~\ell \ne p\\ P^+(\ell-1) = P^+(p-1)}}
\sum_{n =
M+1}^{M+N} \(\frac{su(n)}{\ell p}\), \\
V &= \sum_{\substack{\ell,p\in\sL_z\\P^+(\ell-1)\ne P^+(p-1)}} \sum_{n =
M+1}^{M+N} \(\frac{su(n)}{\ell p}\).
\end{split}
\end{equation*}
To estimate $U$, we use the trivial bound $N$ on each inner sum, deriving that
\begin{equation}
\label{eq:bound U}
U \le  N \sum_{\substack{\ell, p\in \sL_z\\ P^+(\ell-1)  = P^+(p-1)}} 1
\le N \sum_{r\ge z^\alpha}
 \sum_{\substack{\ell, p\in \sL_z\\ \ell \equiv p\equiv 1 \bmod r}} 1
 \ll N \sum_{r\ge z^\alpha}\frac{z^2}{r^2}
 \ll Nz^{2 -\alpha}. 
\end{equation}
To estimate $V$, we first observe that the inequality
$\tau_p(g)\ge P^+(p -1)\ge p^\alpha$  implies (since $\alpha>\tfrac12$)
that $P^+(p-1)\mid \tau_p(g)$ for every $p \in \sL_z$.  

Fix a pair  $(\ell, p) \in \sL_z{\times}\sL_z$ 
with $P^+(\ell-1) \ne P^+(p-1)$, and put
$$
h=\gcd\(\tau_{\ell}(g), \tau_p(g)\)\mand\lambda=g^h.
$$
It is easy to check that $\lambda$ satisfies the conditions of 
Lemma~\ref{lem:Char fgu pl N} with 
$$
\tau_\ell(\lambda)=\tau_\ell(g)/h\ge P^+(\ell-1)\mand
\tau_p(\lambda)=\tau_p(g)/h\ge P^+(p-1).
$$
Furthermore, taking $K = \fl{N/h}$ we obtain that
\begin{equation*}
\begin{split}
\sum_{n = M+1}^{M+N}  \(\frac{su(n)}{\ell p}\) &  = 
 \(\frac{s}{\ell p}\) \sum_{j=1}^h \sum_{k=1}^K \(\frac{u(M + kh +j)}{\ell p}\) + O(h)\\
 &  =  \(\frac{s}{\ell p}\) \
\sum_{j=1}^h \sum_{k=1}^K  \(\frac{f(g^{M +j} \lambda^k)}{\ell p}\) + O(h).
\end{split}
\end{equation*}
Applying Lemma~\ref{lem:Char fgu pl N},  we derive that
\begin{equation}
\label{eq:Prelim} 
\begin{split}
\sum_{n = M+1}^{M+N}  \(\frac{su(n)}{\ell p}\) &  \ll
h \(\frac{K(\ell   p)^{1/2}}{\tau_{\ell p}(\lambda)}+ (\ell   p)^{1/2} \log  (\ell   p) \)  + h\\
&  \ll
\frac{N(\ell   p)^{1/2}}{\tau_{\ell p}(\lambda)}+ h (\ell   p)^{1/2} \log  (\ell   p).
\end{split}
\end{equation}
Hence
\begin{align*}
\sum_{n = M+1}^{M+N}  \(\frac{su(n)}{\ell p}\) 
&  \ll
\frac{N(\ell   p)^{1/2}}{P^+(\ell-1)  P^+(p-1)}+ h (\ell   p)^{1/2} \log  (\ell   p)\\
& \ll  N z^{1- 2 \alpha} + h z \log z.
\end{align*}
Since 
$$
h = \gcd(\tau_{\ell}(g), \tau_p(g))\le  \gcd(\ell-1,p-1), 
$$
we have therefore shown that
\begin{equation}
\label{eq:Bound h}
\sum_{n = M+1}^{M+N}  \(\frac{su(n)}{\ell p}\) \ll  N z^{1- 2 \alpha} + \gcd({\ell}-1, p-1)
 z \log z.
\end{equation}

Combining the bound~\eqref{eq:Bound h} with the definition of $V$, we have
$$
V  \ll N z^{3- 2 \alpha} (\log z)^{-2} +  T  z\log z,
$$
where
$$
T=\sum_{\substack{\ell,p\in\sL_z\\P^+(\ell-1)\ne P^+(p-1)}}\gcd(\ell-1,p-1).
$$
For each pair $(\ell,p)$ in this sum, the primes $\ell_0=P^+(\ell-1)$ and $p_0=P^+(p-1)$
are distinct and satisfy $\min\{\ell_0,p_0\}\ge z^\alpha$. Writing
$\ell-1=\ell_0m_0$ and $p-1=p_0n_0$, it follows that
\begin{equation}
\label{eq:gcd-bd}
\gcd(\ell-1,p-1)\le\min\{m_0,n_0\}\le Cz^{1-\alpha}.
\end{equation}
Thus, using Lemmas~\ref{lem:B-T} and~\ref{lem:Euler recipr} we have
\begin{equation}
\label{eq:T}
T\ll\sum_{m\le Cz^{1-\alpha}}
m \sum_{\substack{p,\ell\in \sL_z\\ \ell\equiv 1\bmod m\\ p\equiv 1\bmod m}}1
\ll  \sum_{m\le Cz^{1-\alpha}} m\(\frac{z}{\varphi(m)\log z}\)^2  \ll z^2(\log z)^{-1}.
\end{equation}
Thus,
\begin{equation}
\label{eq:bound V}
V\ll N z^{3- 2 \alpha} (\log z)^{-2} +  z^3.
\end{equation}

Comparing~\eqref{eq:bound U} and~\eqref{eq:bound V}, 
we see that the bound for $V$ always dominates (since $\alpha < 1$);
hence, as $W=U+V$ we have
$$
W \ll  N z^{3- 2 \alpha} (\log z)^{-2} +  z^3. 
$$
Inserting this bound into~\eqref{eq:newQ} and removing negligible terms,
it follows that
$$
    Q_u(M,N;s)\ll Nz^{1-2\alpha} +z (\log z)^2.
$$
Choosing $z=N^{1/(2\alpha)}(\log N)^{-1/\alpha}$ we obtain the stated result.

\subsection{Proof of Theorem~\ref{thm:QsMN Av1}}

We proceed as in the proof of Theorem~\ref{thm:QsMN}. In particular, 
since the sets $\sN_{s,z}$ are disjoint,  using~\eqref{eq:Ez} and~\eqref{eq:nNsz} 
we derive the following analogue of~\eqref{eq:basicQ}:
$$
\sQ_u(M,N;S)\ll Nz^{-\alpha}+\log z
 +  \frac{(\log
z)^2}{z^2}  \sum_{\substack{s \le S\\ s\text{~sqfree}}} 
 \sum_{n = M+1}^{M+N} 
\left|\sum_{\ell \in \sL_z} \(\frac{su(n)}{\ell}\)\right|^2. 
$$
Thus, extending the summation to all integer $s\in[1,S]$, we obtain 
the following  analogue of~\eqref{eq:newQ}:
\begin{equation}
\label{eq:newQS} 
\sQ_u(M,N;S)\ll Nz^{-\alpha}+\log z+\frac{(\log z)^2}{z^2}\sW. 
\end{equation}
Here, $\sW=\sU+\sV$ with 
\begin{equation*}
\begin{split}
\sU &= \sum_{\substack{\ell, p\in \sL_z,~\ell \ne p\\ P^+(\ell-1) = P^+(p-1)}} 
~\sum_{s \le S}  \sum_{n =M+1}^{M+N} \(\frac{su(n)}{\ell p}\), \\
\sV &= \sum_{\substack{\ell, p\in \sL_z\\  P^+(\ell-1) \ne P^+(p-1)}}
~\sum_{s \le S}   \sum_{n =
M+1}^{M+N} \(\frac{su(n)}{\ell p}\). 
\end{split}
\end{equation*}
We bound $\sU$ trivially as in the proof of Theorem~\ref{thm:QsMN}: 
 \begin{equation}
\label{eq:fU}
\sU \ll  SNz^{2 -\alpha}.
\end{equation}
Next, we estimate $\sV$.   Using~\eqref{eq:Bound h} we have
\begin{equation}
\label{eq:fWUV}
\begin{split}
\sV & \le  \sum_{\substack{\ell, p\in \sL_z\\  P^+(\ell-1) \ne P^+(p-1)}} 
\left| \sum_{s \le S}   \(\frac{s}{\ell p}\)  \right|
\cdot\left| \sum_{n =
M+1}^{M+N} \(\frac{u(n)}{\ell p}\)  \right| \\
&  \ll N z^{1- 2 \alpha} \sV_1 +  z^{1+o(1)}\sV_2\qquad(z\to\infty),
\end{split}
\end{equation}
where
\begin{align*}
\sV_1 & = \sum_{\substack{\ell, p\in \sL_z\\  P^+(\ell-1) \ne P^+(p-1)}} 
\left|\sum_{s \le S}   \(\frac{s}{\ell p}\)\right|,\\
\sV_2 & = \sum_{\substack{\ell, p\in \sL_z\\  P^+(\ell-1) \ne P^+(p-1)}} 
\gcd(\ell-1,p-1)\cdot\left|\sum_{s \le S} 
\(\frac{s}{\ell p}\)\right|. 
\end{align*}

To bound $\sV_1$ we apply the Cauchy inequality and Lemma~\ref{lem:HB},
deriving that
\begin{equation}
\label{eq:fV1}
\sV_1\le\sqrt{z^{2+o(1)}\cdot S(z^2+S)(Sz^2)^{o(1)}}
=(S^{1/2}z^2+Sz)(Sz)^{o(1)}\quad(z\to\infty).
\end{equation}
For $\sV_2$ we bound the sum over $s$ trivially as $O(S)$ and 
use~\eqref{eq:T}, obtaining
\begin{equation}
\label{eq:fV2}
\sV_2  \le Sz^{2+o(1)}\qquad(z\to\infty).
\end{equation}
Inserting the bounds~\eqref{eq:fV1}  and~\eqref{eq:fV2}
into~\eqref{eq:fWUV} we have
\begin{equation}
\label{eq:fV}
\sV\le \(S^{1/2}Nz^{3-2\alpha}+SNz^{2-2\alpha}+Sz^3\)(Sz)^{o(1)}
\qquad(z\to\infty).
\end{equation}
The right side of~\eqref{eq:fU} dominates the 
second term on the right side of~\eqref{eq:fV}; hence, as $\sW=\sU+\sV$, we have
$$
\sW\le \(S^{1/2}Nz^{3-2\alpha}+SNz^{2-\alpha}+Sz^3\)(Sz)^{o(1)}
\qquad(z\to\infty).
$$
After inserting this bound in~\eqref{eq:newQS}
we derive that
\begin{equation}
\label{eq:Qu dirty1}
\sQ_u(M,N;S)\le \(Sz +
S^{1/2}Nz^{1-2\alpha}+SNz^{-\alpha}\)(Sz)^{o(1)}\qquad(z\to\infty)
\end{equation}
(here we have changed the order of terms to make~\eqref{eq:Qu dirty1} readily available 
for an application of Lemma~\ref{lem:grakol} with $J=1$ and $K=2$). 
Noting that~\eqref{eq:Qu dirty1} is  trivial  for $z \le \log N$, 
we now apply Lemma~\ref{lem:grakol} with  $z_1 = \log N$
and with a very large value of $z_2$ (for example   $z_2 = (SN)^{100}$) so 
that the single sums are always dominated by the double sum.
Since $z\to\infty$ as $N\to\infty$, this yields the bound
$$
\sQ_u(M,N;S)\le\(T_{11}+ T_{12} \)
(SN)^{o(1)}\qquad(N\to\infty),
$$
where
\begin{align*} 
T_{11} & =\(S^{2\alpha-1} (S^{1/2} N)\)^{1/(2\alpha)}
= S^{1-1/(4\alpha)}N^{1/(2\alpha)},\\
 T_{12} & = \(S^{\alpha} (SN)\)^{1/(1+\alpha)} = SN^{1/(1+\alpha)}. 
\end{align*}
We also remark that for $S > N$ the result is trivial, so we can replace $(SN)^{o(1)}$
with $N^{o(1)}$, and the proof is complete.

\subsection{Proof of Theorem~\ref{thm:QsMN Av2}}

We proceed as in the proof of Theorem~\ref{thm:QsMN Av1}, 
but we estimate the sum $\sV_2$ in a different way.  To simplify the notation, we now denote
 $g(\ell,p)=\gcd(\ell-1,p-1)$
for all $\ell,p\in\sL_z$.  As in~\eqref{eq:gcd-bd} we have the bound
$$
g(\ell,p)\le Cz^{1-\alpha}
\qquad\(\ell,p\in\sL_z,~P^+(\ell-1)\ne P^+(p-1)\).
$$
Hence, setting $J=\rf{\log(Cz^{1-\alpha})}$ we have
\begin{equation}
\label{V2 V2nu}
\sV_2\le\sum_{\nu=0}^J\sum_{\substack{\ell,p\in\sL_z,~\ell\ne p\\e^\nu<g(\ell,p)\le e^{\nu+1}}}g(\ell,p)
\left|\sum_{\substack{s\le S\\s\text{~sqfree}}}\(\frac{s}{\ell p}\)\right|
\ll\sum_{\nu=0}^J e^\nu\sV_{2,\nu},
\end{equation}
where
$$
\sV_{2,\nu}=\sum_{\substack{\ell,p \in \sL_z,~\ell\ne p \\   g(\ell,p) > e^{\nu}}} 
\left|\sum_{s \le S}   \(\frac{s}{\ell p}\)\right|.
$$
Trivially,
\begin{equation*}
\begin{split}
\sum_{\substack{\ell,p\in\sL_z\\g(\ell,p)>e^{\nu}}}1
\le\sum_{m>e^{\nu}}\sum_{\substack{\ell,p\in\sL_z\\\ell\equiv 1\bmod m\\
p\equiv 1\bmod m}}1
\ll\sum_{m>e^{\nu}}(z/m)^2\ll z^2e^{-\nu}.
\end{split}
\end{equation*}
Hence, using the Cauchy inequality and Lemma~\ref{lem:HB} we see that
$$
\sV_{2,\nu}^2
\ll   z^2 e^{-\nu} 
\sum_{\substack{\ell,p \in \sL_z \\  \ell \ne p }} 
\left|\sum_{s \le S}   \(\frac{s}{\ell p}\)\right|^2
\le z^2e^{-\nu}\cdot (Sz^2 + S^2)(Sz)^{o(1)}
\qquad(z\to\infty),
$$
and therefore
$$
\sV_{2,\nu}\le e^{-\nu/2}(S^{1/2}z^2+Sz)(Sz)^{o(1)}
\qquad(z\to\infty).
$$
Substitution in~\eqref{V2 V2nu} gives 
\begin{equation}
\label{eq:fV2nu}
\sV_2
\le e^{J/2}(S^{1/2}z^2+Sz)(Sz)^{o(1)}
\ll z^{1/2-\alpha/2}(S^{1/2}z^2+Sz)(Sz)^{o(1)}.
\end{equation}
Inserting the bounds~\eqref{eq:fV1}  and~\eqref{eq:fV2nu}
into~\eqref{eq:fWUV} we derive that
\begin{equation}
\begin{split}
\label{eq:fVnu}
\sV\le \(S^{1/2}Nz^{3-2\alpha}+SNz^{2-2\alpha}
+S^{1/2}z^{7/2-\alpha/2}+Sz^{5/2-\alpha/2}\)(Sz)^{o(1)}.
\end{split}
\end{equation}
The right side of~\eqref{eq:fU} dominates the second term  in  the bound~\eqref{eq:fVnu};
hence, recalling that $\sW=\sU+\sV$, we see that
$$
\sW\le \(S^{1/2}Nz^{3-2\alpha}+SNz^{2-\alpha}
+S^{1/2}z^{7/2-\alpha/2}+Sz^{5/2-\alpha/2}\)(Sz)^{o(1)}.
$$
Inserting this bound into~\eqref{eq:newQS} we find that
\begin{equation}
\label{eq:Qu dirty3}
\sQ_u(M,N;S)\le\(S^{1/2}z^{3/2-\alpha/2}+Sz^{1/2-\alpha/2} + SNz^{-\alpha}
+S^{1/2}Nz^{1-2\alpha}\)(Sz)^{o(1)}
\end{equation}
as $z\to\infty$ (we have rearranged terms to make~\eqref{eq:Qu dirty3} readily available 
for an application of Lemma~\ref{lem:grakol} with $J=K=2$). 

As in the proof of Theorem~\ref{thm:QsMN Av2}, we 
 note that~\eqref{eq:Qu dirty3} is  trivial  for $z \le \log N$. 
We now apply Lemma~\ref{lem:grakol} with  $z_1 = \log N$, 
and a very large value of $z_2$, say $z_2 = (SN)^{100}$, so 
that the single sums are dominated by the double sum. This gives
\begin{equation}
\label{eq:Qu clean2}
\sQ_u(M,N;S)\le\(T_{11}+ T_{12} + T_{21} + T_{22}\)
(SN)^{o(1)}\quad(N\to\infty),
\end{equation}
where
\begin{align*} 
T_{11} & =\((S^{1/2})^{\alpha} (S N)^{(3-\alpha)/2}\)^{2/(3+\alpha)} = S^{3/(3+\alpha)}N^{(3-\alpha)/(3+\alpha)},\\
T_{12} & =\((S^{1/2})^{2\alpha-1} (S^{1/2} N)^{(3-\alpha)/2}\)^{2/(1+3\alpha)} = S^{1/2}N^{(3-\alpha)/(1+3\alpha)},\\
 T_{21} & = \(S^{\alpha} (SN)^{(1-\alpha)/2}\)^{2/(1+\alpha)} = SN^{(1-\alpha)/(1+\alpha)}, \\
T_{22} & =\(S^{2\alpha-1} (S^{1/2} N)\)^{2/(3\alpha-1)} = S^{(7\alpha-3)/(6\alpha -2)}N^{(1-\alpha)/(3\alpha-1)}.
\end{align*}
Denoting
$$
\vartheta=\frac{(1-\alpha)^2(1+3\alpha)}{(1+\alpha^2)(3\alpha-1)},
$$
it is straightforward to check that $\vartheta\in(0,1)$
for any $\alpha\in(\tfrac12,1)$, and also 
$$
1-\frac{\vartheta}{2} = \frac{-3+5\alpha+3\alpha^2+3\alpha^3}{(6\alpha-2)(1+\alpha^2)}  
\ge \frac{7 \alpha - 3}{6\alpha-2}. 
$$
Thus, 
$$
T_{22}\le S^{1-\vartheta/2}
N^{(1-\alpha)/(3\alpha-1)}=
T_{12}^\vartheta T_{21}^{1-\vartheta}\le\max\{T_{12},T_{21}\}.
$$
Hence, the term $T_{22}$ can be omitted from~\eqref{eq:Qu clean2}.

Elementary calculations reveal that the inequality 
$T_{21}\ge T_{11}$ cannot hold unless $S\ge N$ (in fact, $S\ge N^{4/(1+\alpha)}$),
in which case $T_{11}\ge N^{(6-\alpha)/(3+\alpha)}\ge N$;
but then our stated bound~\eqref{eq:twoterms} is weaker
than the trivial bound $\sQ_u(M,N;S)\le N$. Consequently, we can
assume $T_{21}\le T_{11}$, and so the term $T_{21}$ can be
omitted from~\eqref{eq:Qu clean2}. Since we can also assume that
$S\le N$ we can replace $(SN)^{o(1)}$ with $N^{o(1)}$ in~\eqref{eq:Qu clean2},
and this completes the proof.

\subsection{Proof of Theorem~\ref{thm:QsMN Av ERH}}
We proceed as in the proof of Theorem~\ref{thm:QsMN Av1},
but in place of $\sL_z$ we use the set $\tsL_z$ consisting
of primes $\ell\in[z,Cz]$ for which
$$
\tau_\ell(g) \ge \ell/\log \ell \qquad \text{and}\qquad P^+(\ell-1) \ge z^\alpha. 
$$
By Corollary~\ref{cor:Mult Ord ERH}, we can assume that
$\left|\tsL_z\right|\ge cz/\log z$.
Another difference is that we estimate the sums
over $s$ directly via Lemma~\ref{lem:ERH}.  

Thus, instead of the bound~\eqref{eq:fU} we obtain 
$$
\sU \le  S^{1/2}Nz^{2 -\alpha+o(1)}\qquad(z\to\infty).
$$
Next, we estimate $\sV$. First we apply Lemma~\ref{lem:ERH} and derive 
$$
\sV  \le  S^{1/2} z^{o(1)} \sum_{\substack{\ell, p\in \tsL_z\\  P^+(\ell-1) \ne P^+(p-1)}} 
\left| \sum_{n =
M+1}^{M+N} \(\frac{u(n)}{\ell p}\)  \right|\qquad(z\to\infty).
$$
Using the equation 
$$
\tau_{\ell p}(\lambda) = \frac{\tau_{\ell}(g) \tau_p(g)}{h^2},
$$
where (as in the proof of of Theorem~\ref{thm:QsMN})
$$
h=\gcd\(\tau_{\ell}(g), \tau_p(g)\) \le \gcd\(\ell-1, p-1\),
$$
and bearing in mind our choice of $\tsL_z$, we can rewrite~\eqref{eq:Prelim} 
as 
\begin{equation*}
\begin{split}
\sum_{n = M+1}^{M+N}  \(\frac{su(n)}{\ell p}\) &  \ll
\frac{N(\ell   p)^{1/2}}{\tau_{\ell p}(\lambda)}+ h (\ell   p)^{1/2} \log  (\ell   p)\\
 &  \ll
\frac{N(\ell   p)^{1/2}h^2 }{\tau_{\ell}(g) \tau_p(g)}+ h (\ell   p)^{1/2} \log  (\ell   p)\\
&  = Nz^{-1+o(1)}  \gcd\(\ell-1, p-1\)^2 +  z^{1+o(1)} \gcd\(\ell-1, p-1\).
\end{split}
\end{equation*}
Hence, we obtain 
\begin{equation}
\label{eq:VQT}
\sV  \le  S^{1/2}  \( NQz^{-1+o(1)}    + Tz^{1+o(1)}\),
\end{equation}
 where 
\begin{align*}
Q& =  \sum_{\substack{\ell, p\in \tsL_z\\  P^+(\ell-1) \ne P^+(p-1)}} 
 \gcd\(\ell-1, p-1\)^2\\
T& =\sum_{\substack{\ell,p\in\tsL_z\\P^+(\ell-1)\ne P^+(p-1)}}\gcd(\ell-1,p-1).
\end{align*}
Clearly, we can still use~\eqref{eq:gcd-bd} with
$\ell, p\in \tsL_z$ with $P^+(\ell-1) \ne P^+(p-1)$, and we have
\eqref{eq:T} as before.  Also,
\begin{equation}
\label{eq:Q}
Q\ll\sum_{m\le Cz^{1-\alpha}}
m^2 \sum_{\substack{p,\ell\in \sL_z\\ \ell\equiv 1\bmod m\\ p\equiv 1\bmod m}}1
\ll  \sum_{m\le Cz^{1-\alpha}} m^2\(\frac{z}{m}\)^2  \ll z^{3-\alpha}.
\end{equation}
Substituting~\eqref{eq:T} and~\eqref{eq:Q} in~\eqref{eq:VQT}, we obtain that
$$
\sV \ll   S^{1/2}  \( N z^{2-\alpha+o(1)}    + z^{3+o(1)}\).
$$
Hence 
\begin{align*}
\sQ_u(M,N;S)&\ll Nz^{-\alpha}+S^{1/2}Nz^{-\alpha+o(1)}+S^{1/2}z^{1+o(1)}\\
&\ll S^{1/2}Nz^{-\alpha+o(1)}+S^{1/2}z^{1+o(1)}.
\end{align*}
Taking $z=N^{1/(1+\alpha)}$ to balance the two terms, we conclude the proof.

\end{document}